\documentclass[a4paper,12pt]{amsart}

\usepackage{amssymb,amsbsy,amsmath,amsfonts,amssymb,amscd}
\usepackage{latexsym}
\usepackage{graphics}
\usepackage{color}
\input xy
\xyoption{all}

\newcommand\sC{{\mathcal C}}

\newcommand\sF{{\mathcal F}}

\newcommand\sB{{\mathcal B}}

\newcommand\LL{{\mathbb L}}

\newcommand\de{\delta}

\DeclareMathOperator{\Ext}{Ext}

\DeclareMathOperator{\divi}{div}

\DeclareMathOperator{\Def}{Def}

\DeclareMathOperator{\ob}{ob}

\newcommand{\CC}{\ensuremath{\mathbb{C}}}

\newcommand{\ZZ}{\ensuremath{\mathbb{Z}}}
\newcommand{\QQ}{\ensuremath{\mathbb{Q}}}

\newcommand{\sS}{\ensuremath{\mathcal{S}}}

\newcommand{\hol}{\ensuremath{\mathcal{O}}}

\newcommand{\PP}{\ensuremath{\mathbb{P}}}

\newcommand{\FF}{\ensuremath{\mathbb{F}}}

\newcommand{\ra}{\ensuremath{\rightarrow}}

\newcommand{\F}{\ensuremath{\mathbb{F}}}

\def\eea{\end{eqnarray*}}
\def\bea{\begin{eqnarray*}}

\newcommand\dual{\mathrel{\raise3pt\hbox{$\underline{\mathrm{\thinspace d
\thinspace}}$}}}
\newcommand\qe{\ifhmode\unskip\nobreak\fi\quad $\Box$}       

\def\BOX{\hfill\lower.5\baselineskip\hbox{$\Box$}}

\newtheorem{theo}{Theorem}[section]
\newtheorem{remarkk}[theo]{Remark}
\newenvironment{rem}{\begin{remarkk}\rm}{\end{remarkk}}

\newtheorem{defin}[theo]{Definition}

\newtheorem{prop}[theo] {Proposition}
\newtheorem{cor}[theo]{Corollary}
\newtheorem{lemma}[theo]{Lemma}
\newtheorem{example}[theo]{Example}

\newtheorem{claim}[theo]{Claim}

\newcommand{\X}{\ensuremath{\mathcal{X}}}

\DeclareMathOperator{\Id}{Id}
\DeclareMathOperator{\Aut}{Aut}

\DeclareMathOperator{\inv}{inv}

\begin{document}

\title[Burniat surfaces II]{ Burniat surfaces II:
   secondary  Burniat surfaces form three connected components of  the 
moduli space.}
\author{I. Bauer, F. Catanese}

\thanks{The present work took place in the realm of the DFG 
Forschergruppe 790 "Classification of algebraic
surfaces and compact complex manifolds".
}

\date{\today}

\maketitle

{\em  This article is  dedicated, with gratitude and admiration,  to 
David Mumford
on the occasion of his $2^3 \cdot 3^2$-th birthday.}

\section*{Introduction}

The so called Burniat surfaces were constructed by Pol Burniat in 1966 
(\cite{burniat}),
where the method of singular bidouble covers was introduced in order to
solve the geography problem for surfaces of general type.

 The special construction of surfaces with geometric genus $p_g(S)=0$,
done in \cite{burniat}, was brought to attention by
\cite{peters}), which  explained Burniat's   calculation of invariants
in the modern language of algebraic geometry, and nowadays the name
of Burniat surfaces is reserved for these surfaces with $p_g(S)=0$.

Burniat surfaces are especially interesting examples for the
non birationality of the bicanonical map ( see \cite{cirobican}). 
For  all the Burniat surfaces $S$ with
$K^2_S \geq 3$ the bicanonical map turns out to be a Galois morphism of degree 4.

We refer to our joint paper with Grunewald and Pignatelli 
(\cite{4names}) for a general
introduction on the classification and moduli problem for surfaces 
with $p_g(S)=0$
and its applications: as an example we mention  our final corollary 
here on the validity of
Bloch's conjecture for all deformations of secondary Burniat surfaces.

The main achievement  of the present  series of three articles is to
completely solve
the moduli problem for Burniat surfaces, determining   the connected
components of the moduli space of  surfaces of general type containing
the Burniat surfaces, and describing their geometry.

The minimal models $S$ of  Burniat surfaces  have as invariant the positive
integer $K^2_S$, which can take
values $K^2_S = 6,5,4,3,2$.

  We get a rationally parametrized family of dimension
$K^2_S -2$ for each value of $K^2_S = 6,5,3,2$, and two such families 
for $K^2_S = 4$,
one called of non nodal type, the other  of nodal type. We proposed in
\cite{burniat1} to call {\em primary Burniat surfaces} those with $K^2_S = 6$,
{\em secondary Burniat surfaces} those with $K^2_S = 5,4$, and
{\em tertiary Burniat surfaces} those with $K^2_S = 3$. The reason not
to consider the Burniat surface with  $K^2_S = 2$ is that it is just 
one special element
of the family of standard Campedelli surfaces (i.e., with torsion group $(\ZZ / 2Ê\ZZ)^3$) (see \cite{kulikov} and 
\cite{burniat1}), whose geometry is completely
understood (see \cite{miyaoka} and \cite{miles}).

An important result was obtained by Mendes Lopes and Pardini in \cite{mlp}
who proved that primary Burniat surfaces form a connected component 
of the moduli space
of surfaces of general type.  A stronger result concerning primary Burniat
surfaces was proved in part one (\cite{burniat1}), namely that any
surface homotopically equivalent
to a primary Burniat surface is a primary Burniat surface.  Alexeev 
and Pardini (cf. \cite{alpar})
reproved the result of Mendes Lopes and Pardini by studying more generally
the component of the moduli space of stable surfaces of general type 
containing primary
Burniat surfaces.

Here, we shall prove in one go that each of the 4 families of Burniat surfaces
with $K^2_S  \geq 4$, i.e., of primary and secondary Burniat surfaces, is a
connected component  of the moduli space
of surfaces of general type.

The case of tertiary Burniat surfaces will be treated in the third 
one of this series of papers,
and we limit ourselves here to say that the general deformation of a 
Burniat surface
with $K^2_S = 3$ is not a Burniat surface, but it is always a 
bidouble cover
(through the bicanonical map) of a cubic surface with three nodes.

At the moment when we started the redactional work for the present 
paper we became aware of the fact that
a weaker result was stated in \cite{kulikov}, namely that each family of 
Burniat surfaces of
secondary type yields a dense set in an irreducible component of the 
moduli space.
The result is derived by Kulikov from the assertion that the base of 
the Kuranishi family of
deformations is smooth. This result  is definitely false for
the Burniat surfaces of nodal type (proposition 4.12 and corollary 
4.23 (iii) of \cite{kulikov}),
  as we shall now see.

Indeed one of the main technical contributions of this paper is the 
study of the deformations
of secondary Burniat surfaces, through diverse techniques.

A very surprising and new phenomenon occurs for nodal surfaces, 
confirming Vakil's
`Murphy's law' philosophy (\cite{murphy}). To explain it, recall that 
indeed there are two
diffferent scheme structures for the moduli spaces of surfaces of general type.

One is the moduli space
$\mathfrak M^{min}_{\chi, K^2}$ for minimal models $S$ having
$\chi(\hol_S) = \chi$, $K^2_S = K^2$, the other is the Gieseker 
moduli space (\cite{gieseker})
$\mathfrak M^{can}_{\chi, K^2}$ for canonical models $X$ having
$\chi(\hol_X) = \chi$, $K^2_X = K^2$.  Both are quasi projective 
schemes and  there is a natural morphism
$\mathfrak M^{min}_{\chi, K^2} \ra \mathfrak M^{can}_{\chi, K^2}$ which is a
bijection. Their local structure as
complex analytic spaces is the quotient of the base of the Kuranishi
family by the action
of the finite group $ \Aut(S) = \Aut(X)$.

In  \cite{enr} series of examples were exhibited where $\mathfrak 
M^{can}_{\chi, K^2}$ was smooth,
but  $\mathfrak M^{min}_{\chi, K^2}$ was everywhere non reduced.

For nodal Burniat surfaces with $K^2_S = 4$ both spaces are 
everywhere non reduced,
but the nilpotence order is higher for  $\mathfrak M^{min}_{\chi, 
K^2}$; this is
a further pathology, which adds to the ones presented in  \cite{enr} 
and in \cite{murphy}.

More precisely, this is one of our two main results:

\begin{theo}\label{connnodal} The subset  of the Gieseker moduli space
$\mathfrak M^{can}_{1, 4}$ of canonical
surfaces of general type  $X$  corresponding to
Burniat surfaces $S$ with $K_S^2 = 4$ and of  nodal type
is an irreducible connected component of dimension 2, rational and
  everywhere non reduced.

More precisely, the base of the Kuranishi family of $X$ is locally 
analytically isomorphic
to $\CC^2 \times Spec (\CC[t]/ (t^m))$, where $m$ is a fixed integer, 
$m \geq 2$.

The corresponding subset  of the  moduli space $\mathfrak M^{min}_{1, 
4}$ of minimal
surfaces $S$ of general type is also   everywhere non reduced.

More precisely, the base of the Kuranishi family of $S$ is locally 
analytically isomorphic
to $\CC^2 \times Spec (\CC[t]/ (t^{2m}))$.

\end{theo}

Whereas for the non nodal case we get the following second main result:

\begin{theo}\label{conn} The three respective subsets  of the  moduli 
spaces of minimal
surfaces of general type $\mathfrak M^{min}_{1, K^2}$   corresponding to
Burniat surfaces with $K^2 = 6$, resp. with $K^2 = 5$, resp.
Burniat surfaces with $K^2 = 4$ of non nodal type,
are irreducible connected components, normal, rational
  of respective dimensions 4,3,2.

Moreover, the base of the Kuranishi family of such surfaces $S$ is smooth.
\end{theo}

Theorem \ref{connnodal}  poses the  challenging deformation theoretic 
question to calculate the number $m$
giving  the order of nilpotence of the local moduli space (and also 
of the moduli space at the general point).

\bigskip

\noindent
{\bf Acknowledgements.}
The second author  would like to  express his gratitude to David Mumford,
for all that he learnt reading his penetrating books and articles, 
and especially for
his   hospitality  and mathematical guidance during his
  visit to  Harvard University in the year 1977-1978.

Thanks also to Rita Pardini for spotting a mistake in a previous version of
proposition \ref{gorenstein}.

\begin{figure}[htbp]\label{configs}
\begin{center}
\scalebox{0.7}{\includegraphics{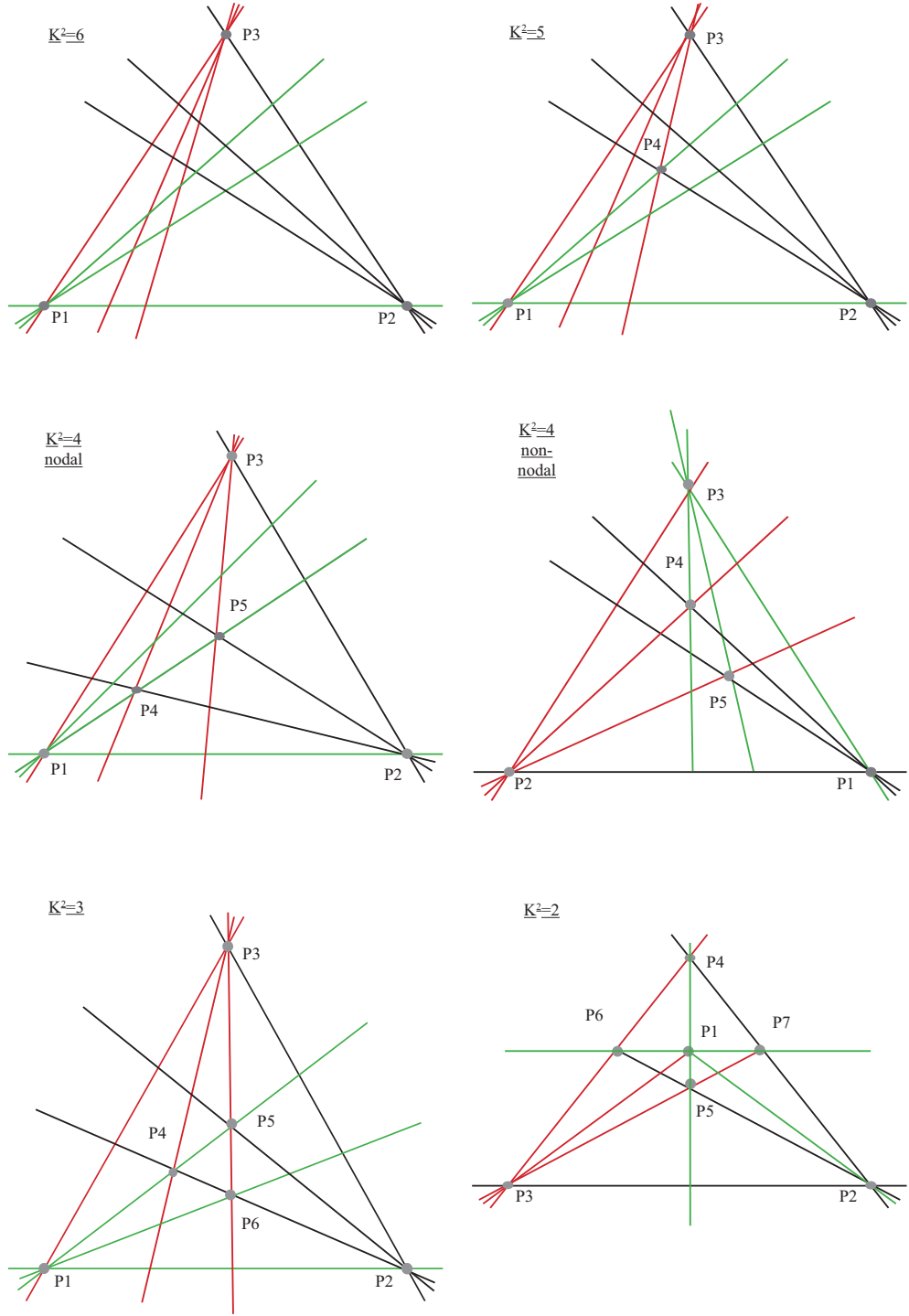}}
\end{center}
\caption{Configurations of lines}
\end{figure}

\section{The local moduli spaces of Burniat surfaces}\label{locmod}

Burniat surfaces are minimal surfaces of general type with $K^2 
=6,5,4,3,2$ and $p_g = 0$, which were
constructed in \cite{burniat} as singular bidouble covers (Galois 
covers with group $(\ZZ/2\ZZ)^2$) of the
projective plane branched on 9 lines.

We briefly recall their construction: this will also be useful to fix 
our notation.
  For more details, and for the proof that Burniat surfaces are 
exactly certain Inoue surfaces we refer to
\cite{burniat1}.

\medskip
\noindent Let $P_1, P_2, P_3 \in \PP^2$ be three non collinear points 
(which we assume to be the points
$(1:0:0)$, $(0:1:0)$ and $(0:0:1)$)  and let's denote by 
$Y:=\hat{\PP}^2(P_1, P_2,P_3)$ the Del Pezzo surface of
degree $6$,  blow up of $\PP^2$ in $P_1, P_2, P_3$.

   $Y$ is `the' smooth Del Pezzo surface of degree $6$, and it is the 
closure of the graph of the rational map
$$\epsilon: \PP^2 \dashrightarrow \PP^1 \times \PP^1 \times \PP^1$$ 
such that $$\epsilon (y_1 : y_2: y_3)  =
((y_2 : y_3) ( y_3: y_1) ( y_1: y_2)).$$

One sees immediately that $Y \subset \PP^1 \times \PP^1 \times \PP^1$ 
is the hypersurface of type $(1,1,1)$:
$$ Y = \{(( x_1' : x_1), (  x_2':  x_2), ( x_3':  x_3)) \ | \ x_1 x_2 
x_3 = x_1' x_2' x_3' \}.$$

We denote by $L$ the total transform of a general line in $\PP^2$, by
$E_i$ the exceptional curve lying over $P_i$,   and by $D_{i,1} $ the 
unique effective divisor in $ |L - E_i-
E_{i+1}|$, i.e., the proper transform of the line $y_{i-1} = 0$, side 
of the triangle joining the points $P_i, P_{i+1}$.

Consider on $Y$,  for each   $  i \in  \ZZ/3\ZZ \cong \{1,2,3\}$,
  the following divisors
$$ D_i = D_{i,1} + D_{i,2} + D_{i,3} + E_{i+2} \in |3L - 3E_i - 
E_{i+1}+E_{i+2}|,$$

where $D_{i,j} \in |L - E_i|, \ \rm{for} \ j = 2,3, \  D_{i,j} \neq 
D_{i,1}$, is the proper transform of another  line
through
$P_i$ and $D_{i,1} \in |L - E_i- E_{i+1}|$ is as above. Assume also 
that all the corresponding lines in $\PP^2$ are
distinct, so that $D : = \sum_i D_i$ is  a reduced divisor.

Note that, if we define  the divisor $\mathcal{L}_i : = 3L - 2 
E_{i-1} - E_{i+1}$, then
$$D_{i-1} + D_{i+1} = 6L - 4 E_{i-1} - 2E_{i+1} \equiv 2
\mathcal{L}_i,$$ and we can consider (cf. \cite{sbc}) the associated 
bidouble cover $X' \rightarrow Y$ branched
on $D : = \sum_i D_i$ (but we take a different ordering of the 
indices of the fibre coordinates $u_i$, using the
same choice as the one made in \cite{burniat1}, where however $X'$ 
was denoted by $X$).

We recall that this precisely means the following: let $D_i =
\divi(\delta_i)$, and let $u_i$ be
     a fibre coordinate of the geometric line bundle $\LL_{i+1}$, 
whose sheaf of holomorphic sections is
$\hol_Y(\mathcal{L}_{i+1})$.

Then $X \subset \LL_1 \oplus \LL_2 \oplus
\LL_3$ is given by the equations:
$$ u_1u_2 = \delta_1 u_3, \ \ u_1^2 = \delta_3 \delta_1;
$$
$$ u_2u_3 = \delta_2 u_1, \ \ u_2^2 = \delta_1 \delta_2;
$$
$$ u_3 u_1 = \delta_3 u_2, \ \ u_3^2 = \delta_2 \delta_3.
$$

    From the birational point of view, as done by Burniat, we are 
simply adjoining to the function field of $\PP^2$
two square roots, namely $\sqrt  \frac{\Delta_1}{ \Delta_3}$ and
$\sqrt  \frac{\Delta_2}{ \Delta_3}$, where $\Delta_i$ is the cubic 
polynomial in $\CC[x_0,x_1,x_2]$ whose zero
set has $D_i - E_{i+2}$ as strict transform.

This shows clearly that we have a Galois cover $X' \ra Y$ with group 
$(\ZZ / 2 \ZZ)^2$.

The equations above give a biregular model $X'$ which is nonsingular 
exactly if the divisor $D$ does not have
points of multiplicity 3 (there cannot be points of higher 
multiplicities). These points give then quotient
singularities of type $\frac{1}{4} (1,1) $, i.e., isomorphic to the 
quotient of $\CC^2$ by the action of $(\ZZ / 4
\ZZ)$ sending
$ (u,v) \mapsto (iu, iv)$ (or, equivalently , the affine cone over 
the 4-th Veronese embedding of $\PP^1$).

\begin{defin}
     A {\em primary Burniat surface} is a surface constructed as 
above, and which is moreover smooth.
     It is then a minimal surface $S$ with $K_S$ ample, and with 
$K_S^2 = 6$, $p_g(S) =q(S)= 0$.

      A  {\em secondary Burniat surface} is the minimal resolution of 
a surface $X'$ constructed as above, and which
moreover has
$ 1 \leq m \leq 2$ singular points (necessarily of the type described above).
     Its minimal resolution is then a minimal surface $S$ with $K_S$ 
nef and big, and with $K_S^2 = 6-m$,
$p_g(S) =q(S)= 0$.

     A {\em tertiary (respectively, quaternary ) Burniat surface } is 
the minimal resolution of a surface $X'$
constructed as above, and which  moreover has $ m = 3 $ (respectively 
$ m=4$) singular points (necessarily of the
type described above).
     Its minimal resolution is then a minimal surface $S$ with $K_S$ 
nef and big,
but not ample,  and with $K_S^2 = 6-m$,
$p_g(S) =q(S)= 0$.

\end{defin}

\begin{rem} 1) We remark that for $K_S^2 =4$ there are two possible 
types of configurations. The  one where
there are three collinear points of multiplicity at least 3 for the 
plane curve formed by the 9 lines leads to a
Burniat surface $S$ which we call of  {\em nodal type}, and with 
$K_S$ not ample, since the inverse image of the
line joining the 3 collinear points is a (-2)-curve (a smooth 
rational curve of self intersection $-2$).

   In the other cases with $K_S^2 =4, 5, 6$, instead, $K_S$ is ample.

2) In the nodal case,  if we  blow up  the two $(1,1,1)$ points of
$D$, we obtain a weak Del Pezzo surface $\tilde{Y}$, since it 
contains a (-2)-curve. Its anticanonical model $Y'$
has a node (an
$A_1$-singularity, corresponding to the contraction of the 
(-2)-curve). In the non nodal case, we obtain a
smooth Del Pezzo surface $\tilde{Y} = Y'$ of degree $4$.

\noindent 3) We illustrated the possible configurations of the lines in 
the plane in figure
1.
\end{rem}

We will mostly restrict ourselves in the following
  to {\em secondary Burniat surfaces}. Therefore the branch divisor 
$D$ on $Y$ has one ($P_4$) or two singular
points ($P_4, P_5$) of type
$(1,1,1)$ according to whether we are in the case $K_S^2 = 5$ or in 
the case $K_S^2 = 4$.

Since looking at the graphical picture might not be sufficient, we 
describe  our  situation  also through
an appropriate mathematical notation. Let $\tilde{Y} \ra Y$ be the 
blow up of $Y$ in $P_4$ (if
$K_S^2 = 5$), respectively in the points
$P_4, P_5$ (if
$K^2 = 4$). Let $E_4$ (resp. $E_5$) be the exceptional curve lying 
over $P_4$ (resp. over $P_5$).

We have summarized  in the tables (\ref{K25}), (\ref{K24nn}), 
(\ref{K24n}) the linear equivalence classes of the
divisors  $D_{i,j}$, which are the strict transforms of lines 
$D_{i,j}'$ in $\PP^2$.

\begin{small}
\begin{table}[ht]
\caption{$K_S^2 = 5$}
\label{K25}
\begin{tabular}{|c|c|}
\hline
$(i,j)$  & $D_{i,j}$\\
\hline\hline
$(i,1)$ &  $L-E_i - E_{i+1}$ \\
   $(i,2)$ &  $L-E_i - E_4$ \\
  $(i,3)$ &  $L-E_i$ \\
\hline
\end{tabular}
\end{table}
\end{small}

\begin{small}
\begin{table}[ht]
\caption{$K_S^2 = 4$: non nodal}
\label{K24nn}
\begin{tabular}{|c|c|}
\hline
$(i,j)$  & $D_{i,j}$\\
\hline\hline
$(i,1)$ &  $L-E_i - E_{i+1}$ \\
   $(i,2)$ &  $L-E_i - E_4$ \\
   $(i,3)$ &  $L-E_i - E_5$ \\
\hline
\end{tabular}
\end{table}
\end{small}
\begin{small}

\begin{table}[ht]
\caption{$K_S^2 = 4$:  nodal}
\label{K24n}
\begin{tabular}{|c|c|c|}
\hline
$(i,j)$  & $D_{i,j}$\\
\hline\hline
  $(i,1)$ &  $L-E_i - E_{i+1}$ \\
   $(1,2)$ &  $L-E_1 - E_4 -E_5$ \\
   $(1,3)$ &  $L-E_1$ \\
   $(2,2)$ &  $L-E_2 - E_4$ \\
  $(2,3)$ &  $L-E_2 - E_5$ \\
$(3,2)$ &  $L-E_3 - E_4$ \\
$(3,3)$ &  $L-E_3 - E_5$ \\
\hline \hline
\end{tabular}
\end{table}
\end{small}

We see easily from tables (\ref{K25}), (\ref{K24nn}), (\ref{K24n}) 
some formulae which hold uniformly for all
Burniat surfaces:

\begin{rem}\label{3prop} i) $D_i \equiv - K_{\tilde{Y}} - 2 E_{i} + 2 E_{i+2} $

\noindent ii) $\mathcal{L}_i \cong \hol_{\tilde{Y}} (- K_{\tilde{Y}} 
+ E_{i} -  E_{i-1} )$ since  $\mathcal{L}_i
\cong \hol_{\tilde{Y}} ( L_i) $, where
$L_i \equiv  \frac 12 (D_{i-1} + D_{i+1}$.

This yields $L _i \equiv  3L -E_{i+1} - 2 E_{i-1} - E_4$ for $K^2 = 
5$,  and  $\L_i = 3L -E_{i+1} - 2 E_{i-1} -
E_4-E_5$, for $K^2 = 4$.

\noindent iii) $D_i - L_i \equiv - 3 E_i + 3 E_{i-1}$.

\end{rem}

We have the following

\begin{lemma}\label{h1}
  Let $L_i$, $i=1,2,3$ be as in the above remark.
  Then $H^1(\tilde{Y},\hol_{\tilde{Y}}(-L_i)) =0$.
\end{lemma}

\begin{proof}
  If one has a reduced connected curve $C$ on a  surface $Y$ with $H^1 
(\hol_Y) = 0$, then necessarily $H^1
(\hol_Y(-C)) = 0$.

Since $L_i \equiv - K_{\tilde{Y}} + E_{i} -  E_{i-1}$ we easily find 
this divisor $C$ in the linear subsystem
$$D_{i-1,2} + D_{i-1,3} + |L - E_{i+1}|,$$ taking  a general line in 
$ |L - E_{i+1}|$.

\end{proof}

\begin{rem}\label{DminusL} From the long exact cohomology sequence 
associated to the short eaxact sequence
of sheaves
$$ 0 \ra \hol_{\tilde{Y}}(-L_i) \ra \hol_{\tilde{Y}}(D_i -L_i) \ra 
\hol_{D_i} (D_i -L_i) \ra 0,
$$ and lemma \ref{h1} it follows that
$$H^0(\tilde{Y}, \hol_{\tilde{Y}}(D_i -L_i)) \cong H^0(D_i, 
\hol_{D_i} (D_i -L_i)),$$ for all $i \in \{1,2,3\}$.
\end{rem}

We refer to \cite{ms}, def. 2.8, page 494 , and to \cite{sbc}, p. 106 
for the definition of the family of natural
deformations of a bidouble cover.

\begin{prop}
  Let $S$ be the minimal model of a  Burniat surface, given
  as Galois $(\ZZ/2\ZZ)^2$-cover of the (weak) Del Pezzo surface 
$\tilde{Y}$. Then all natural deformations of
$\pi \colon S \ra \tilde{Y}$ are Galois $(\ZZ/2\ZZ)^2$-covers of $\tilde{Y}$.
\end{prop}

\begin{proof} The natural deformations of a bidouble cover are 
parametrized by the direct sum of the vector
spaces
$H^0(\tilde{Y}, \hol_{\tilde{Y}}(D_i )) $ with  the vector spaces
$H^0(\tilde{Y}, \hol_{\tilde{Y}}(D_i -L_i)) $. The second summand is 
zero exactly when all the natural
deformations are Galois.

  As observed in iii) of remark \ref{3prop} in all the  cases we have
$$ D_i - L_i \equiv -3E_i + 3E_{i+2}, \ \forall i \in \{1,2,3\}
$$ Assume that there exists an effective divisor  $C \in |D_i - 
L_i|$. Then $C \cdot E_{i+2} = -3$,
  whence $C \geq 3E_{i+2}$. Therefore we can write $C = C' + 
3E_{i+2}$, with $C' \in |-3E_i|$, a contradiction.
This implies that $|D_i - L_i| = \emptyset$.

\end{proof}

\begin{rem}\label{dimensions}
  It is easy to see that the respective dimensions of the  families of 
Burniat surfaces are
\begin{itemize}
  \item[-] 4 for $K^2 = 6$;
  \item[-] 3 for $K^2 = 5$;
\item[-]  2 for $K^2 = 4$, non nodal;
\item[-] 2 for $K^2 = 4$, nodal.
\item[-] 1 for $K^2 = 3$.
\end{itemize}
\end{rem}

An important feature of each family of Burniat surfaces is that the 
canonical models do not get worse
singularities for special elements of the family.

The minimal model  $S$ of a  Burniat surface is a smooth bidouble 
cover of a smooth weak Del Pezzo surface
$\tilde{Y}$, branched over a normal crossings divisor.
$K_S$ is ample for $K^2_S \geq 4$ unless we are in the nodal case 
with  $K^2_S = 4$.

In this nodal case one has a singular Del Pezzo surface $Y'$ with an 
$A_1$-singularity obtained contracting the
(-2) curve $D_{1,2}$.

The canonical model $X$ of $S$ is obtained contracting the (-2) curve 
$E$ which is the inverse image of
$D_{1,2}$. $X$ is a finite bidouble cover of $Y'$.

In this last case we shall preliminarly investigate numerical 
invariants of the  Kuranishi family of $S$ and then
use them to describe the
  Kuranishi family of $X$ .

Our first goal will be to determine
$\dim H^1(S, \Theta_S) $, using the following special case of theorem 
2.16 of \cite{ms} :

\begin{prop} Let $\pi : S \ra \tilde{Y}$ be a Galois $(\ZZ / 2 
\ZZ)^2$-cover of smooth projective surfaces with
branch divisor $D:= D_1 + D_2 + D_3$. Then
\begin{multline*}
\pi _*(\Omega_S^1 \otimes \Omega_S^2) = (\Omega_{\tilde{Y}}^1 (\log 
D_1, \log D_2, \log D_3)
\otimes \Omega_{\tilde{Y}}^2 ) \oplus \\
\oplus (\bigoplus_{i=1}^3 \Omega_{\tilde{Y}}^1 (\log D_i) \otimes 
\Omega_{\tilde{Y}}^2 \otimes
\hol_{\tilde{Y}}(L_i)),
\end{multline*} where $\Omega_{\tilde{Y}}^1 (\log D_1, \log D_2, \log 
D_3)$ is the subsheaf of  the sheaf of
rational 1-forms generated by $\Omega_{\tilde{Y}}^1$ and by 
$dlog(\de_1), dlog(\de_2), dlog(\de_3)$, and
where $D_i = div (\de_i)$.

Moreover the first summand is the invariant one, and the other three 
correspond to the
three non trivial characters of $(\ZZ / 2 \ZZ)^2$.

\end{prop}

We are able to use the above result observing in fact that the sheaf 
$\Omega_S^1 \otimes
\Omega_S^2 = \Omega_S^1 (K_S)$ is the Serre dual of $\Theta_S$,  and 
that for each locally free sheaf $\sF$ on
$S$  we have (the second formula is duality for a finite map, cf. 
\cite{hart}, exercise 6.10, page 239):

\begin{itemize}
\item
$H^i(\sF) = H^i(\pi_* (\sF))$,
\item
$\pi_* (\sF^{\vee}(K_S) )\cong (\pi_* \sF)^{\vee}(K_{\tilde{Y}}) $,
\item
$K_S  = \pi^*(K_{\tilde{Y}} + L_1 + L_2 + L_3)$
\item
$H^i(\Theta_S)^{\vee} = H^{2-i}(\pi_* (\Omega_S^1 \otimes
\Omega_S^2))$.

\end{itemize}

Moreover, we use the following exact residue sequence

$$ 0 \ra \Omega^1_{\tilde{Y}} \ra \Omega^1_{\tilde{Y}}(\log D_1, 
\dots ,\log D_k) \ra
\bigoplus_{i=1}^k \hol_{D_i} \ra 0
$$ holding more generally if the divisors $D_i$ are reduced and $ 
\tilde{Y}$ is a factorial variety (see e.g. lemma
3, page 675 of \cite{mld}).

We are  left with the calculation of the  cohomology groups of the  sheaves:

$$ \Omega_{\tilde{Y}}^1 (\log D_1, \log D_2, \log D_3) 
(K_{\tilde{Y}}) ,$$ respectively
$$ \Omega_{\tilde{Y}}^1 (\log D_i)  (K_{\tilde{Y}} + L_i) .$$

However, the second cohomology groups vanish since $S$ is of general 
type hence $H^0(\Theta_S)=0$. The
Riemann Roch theorem tells us what are the alternating sums of the 
dimensions, thus in the end it suffices to
calculate the $H^0$ of these sheaves.

Let us look at the invariant part, using the exact sequence
$$ 0 \ra \Omega^1_{\tilde{Y}} (K_{\tilde{Y}})\ra 
\Omega^1_{\tilde{Y}}(\log D_1, \log D_2 ,\log D_3)
(K_{\tilde{Y}}) \ra
\bigoplus_{i=1}^3 \hol_{D_i} (K_{\tilde{Y}})\ra 0.
$$

The space $H^0 (\Omega^1_{\tilde{Y}} (K_{\tilde{Y}}) )$ vanishes since
$H^0 (\Omega^1_{\tilde{Y}})= 0$ and $- K_{\tilde{Y}}$ is effective.

Moreover, if $\tilde{Y}$ is a Del Pezzo surface, then $- 
K_{\tilde{Y}}$ is ample and also $H^0 (\hol_{D_i}
(K_{\tilde{Y}}) ) =0$.

Thus $H^0 (\Omega^1_{\tilde{Y}}(\log D_1, \log D_2 ,\log D_3) 
(K_{\tilde{Y}})  ) = 0$ unless we are in the
nodal case. Here there is the (-2) curve $D_{1,2}$ which  is a 
connected component of $D_1$, hence in this case
$H^0 (\hol_{D_1} (K_{\tilde{Y}}) ) \cong \CC.$

On the other hand the coboundary in the long exact cohomology 
sequence is given by cup product with the
extension class, which is the direct sum of the Chern classes of the 
divisors $D_i$, $c_1 (D_i) \in H^1
(\Omega^1_{\tilde{Y}})$.

Note, in the nodal case,  that $|-K_{\tilde{Y}}| = |3L - \sum_{i=1}^5 
E_i|$ is base point free.  Therefore there is a
morphism
$\hol_{\tilde{Y}}(K_{\tilde{Y}}) \ra \hol_{\tilde{Y}}$, which is not 
identically zero on any component of the
$D_i$'s.

We get the commutative diagram with exact rows
\begin{small}
\begin{equation}\label{diagram2}
\xymatrix{ 0\ar[r]& \Omega^1_{\tilde{Y}}(K_{\tilde{Y}}) \ar[d] 
\ar[r]&\Omega^1_{\tilde{Y}} (\log D_1, \log
D_2 ,\log D_3)(K_{\tilde{Y}})\ar[d]\ar[r] &\oplus_{i=1}^3 
\hol_{D_i}(K_{\tilde{Y}})
\ar[d] \ar [r]&0\\ 0\ar[r]&\Omega^1_{\tilde{Y}} 
\ar[r]&\Omega^1_{\tilde{Y}}(\log D_1, \log D_2 ,\log D_3)
\ar[r] &\oplus_{i=1}^3
\hol_{D_i}  \ar [r]&0.}
\end{equation}
\end{small}
 From this we get the commutative diagram
\begin{equation*}
\xymatrix{
\CC \cong H^0(\tilde{Y},\oplus_{i=1}^3 
\hol_{D_i}(K_{\tilde{Y}}))\ar[d]^{\psi_2}\ar[r]^{\delta}
\ar[dr]^{\varphi}& H^1(\tilde{Y}, 
\Omega^1_{\tilde{Y}}(K_{\tilde{Y}})) \ar[d] \\ 
H^0(\tilde{Y},\oplus_{i=1}^3
\hol_{D_i})\ar[r]^{\psi_1} &H^1(\tilde{Y}, \Omega^1_{\tilde{Y}}) .}
\end{equation*} Note that by a straightforward extension of the 
argument given in \cite{ms}, lemma 3.7,  the
image of the function identically equal to $1$ on $D_{1,2}$ maps 
under $\psi_1$ to the first  Chern class of
$D_{1,2}$. Hence $\varphi = \psi_1 \circ \psi_2 \neq 0$, hence also 
$\delta$ is non zero.

We have thus accomplished the proof of

\begin{lemma} For a primary or secondary Burniat surface the 
$G:=(\ZZ/2\ZZ)^2$-invariant part
$H^0 (\Omega_S^1 \otimes \Omega_S^2))^G$ of $H^0 (\Omega_S^1 \otimes 
\Omega_S^2))$ vanishes.
\end{lemma}

Let us now turn to the other characters.

We have then the other exact sequence

$$ 0 \ra \Omega^1_{\tilde{Y}} (K_{\tilde{Y}} + L_i)  \ra 
\Omega_{\tilde{Y}}^1 (\log D_i)  (K_{\tilde{Y}} + L_i)
\ra
\hol_{D_i}(K_{\tilde{Y}} + L_i)  \ra 0
$$ and we recall that, by remark \ref{3prop}

$$\Omega_{\tilde{Y}}^1 (\log D_i)(K_{\tilde{Y}} + L_i) \cong 
\Omega_{\tilde{Y}}^1 (\log D_i) (E_i - E_{i+2})
.$$

We shall calculate the dimension of the space $$H^0 
(\Omega_{\tilde{Y}}^1 (\log D_i) (E_i - E_{i+2})  )$$  taking
the direct image sheaf on $\PP^2$.

We need a lemma which we state  for simplicity in the
case of dimension 2: it shows what is the efffect of blowing down a (-1) curve.

\begin{lemma} Consider a finite set of distinct linear forms 
$$l_{\alpha} : =  y - c_{\alpha} x,
\alpha \in A$$ vanishing at the origin in $\CC^2$. Let $p \colon Z 
\ra \CC^2$  be the blow up of
  the origin, let $D_{\alpha} $ be the strict transform of the line
$L _{\alpha} : = \{ l_{\alpha} = 0 \}$, and let $E$ be the exceptional divisor.

Let  $\Omega_{\CC^2}^1 ((d\log l_{\alpha})_{\alpha \in A})$ be the 
sheaf of rational 1-forms $\eta$ generated
by  $\Omega_{\CC^2}^1$ and by the differential forms
$d \log l_{\alpha}$ as an $\hol_{\CC^2}$-module and define similarly
$ \Omega^1_{Z}((\log D_{\alpha})_{\alpha \in A}) $. Then:

\begin{enumerate}
\item
$p_* \Omega^1_{Z}(\log E)(-E) = \Omega_{\CC^2}^1$,
\item
$p_* \Omega^1_{Z}(\log E, (\log D_{\alpha})_{\alpha \in A}) =  p_* 
\Omega^1_{Z}( (\log D_{\alpha})_{\alpha
\in A})(E) = $ 

$\Omega_{\CC^2}^1 ((d\log l_{\alpha})_{\alpha \in A})$,
\item
$p_* \Omega^1_{Z}((\log D_{\alpha})_{\alpha \in A}) =
  \{ \eta \in \Omega_{\CC^2}^1 ((d\log l_{\alpha})_{\alpha \in A})|
\eta =$

$=\sum_{\alpha} g_{\alpha} d \log l_{\alpha} + \omega,
\omega \in \Omega_{\CC^2}^1, \sum_{\alpha} g_{\alpha}(0) = 0\}$.
\end{enumerate}
\end{lemma}

\begin{proof} The sheaf  $\Omega_{\CC^2}^1 ((d\log 
l_{\alpha})_{\alpha \in A})$ is locally free  outside of the
origin, and torsion free in view of the residue sequence, since 
$\bigoplus_{\alpha \in A} \hol_{L_{\alpha}}$ has
no section with a $0$-dimensional support.

Likewise, all other direct image sheaves are torsion free, and those 
in 2. and 3. are equal to  $\Omega_{\CC^2}^1
((d\log l_{\alpha})_{\alpha \in A})$ outside of the origin.

1.: $p_* \Omega^1_{Z}(\log E)(-E) \subset \Omega_{\CC^2}^1$ holds 
since the left hand side is torsion free and
coincides with the right hand side outside the origin. But 
$\Omega_{\CC^2}^1$ is locally free, hence it enjoys
the Hartogs property, so the desired inclusion holds. It suffices 
then to show that $p^* ( \Omega_{\CC^2}^1)
\subset \Omega^1_{Z}(\log E)(-E)$. This follows since in the affine 
chart $ (x,t) \mapsto (x, y = xt)$ of the blow
up, we have $ dx = x d\log x$, $ dy = x (dt + t d\log x)$ (and 
similarly on the other chart).

2.: it suffices to show the chain of inclusions (where $m \geq 1$)

\begin{multline*}
\Omega_{\CC^2}^1 ((d\log l_{\alpha})_{\alpha \in A})
\subset  p_* \Omega^1_{Z}(\log E, (\log D_{\alpha})_{\alpha \in A}) \subset \\
\subset p_* \Omega^1_{Z}( (\log D_{\alpha})_{\alpha \in A})(mE)
\subset  \Omega_{\CC^2}^1 ((d\log l_{\alpha})_{\alpha \in A}).
\end{multline*}

The first inclusion follows, the two sheaves being torsion free and 
equal outside of the origin,
from the  assertion that
$ p^* (\Omega_{\CC^2}^1 ((d\log l_{\alpha})_{\alpha \in A})) \subset
\Omega^1_{Z}(\log E, (\log D_{\alpha})_{\alpha \in A}) $.

This assertion is easily verified in each affine chart, since
$  d \log  l_{\alpha} =   d \log  x  +  d \log  (\frac{l_{\alpha}}{x} 
)=  d \log  x  +  d \log  (t - c_{\alpha} )$.

The second inclusion is obvious, while,  for the third,
$$ p_* \Omega^1_{Z}( (\log D_{\alpha})_{\alpha \in A})(mE) $$ 
consists of rational differential 1-forms
$\omega$ which, when restricted to $\CC^2 \setminus \{0\}$, yield sections of
$\Omega_{\CC^2}^1 ((d\log l_{\alpha})_{\alpha \in A}) $.

Therefore in particular $\omega \prod_{\alpha \in A}  l_{\alpha} $ is 
a regular holomorphic 1-form on $\CC^2$.

Thus, modulo holomorphic 1-forms, we can write
$$ \omega = \frac{f } {\prod_{\alpha \in A}  l_{\alpha}} dx +
\frac{g } {\prod_{\alpha \in A}  l_{\alpha}} dy,$$ where $f,g$ are 
pseudopolynomials of degree in $y$ less than
$r : = {\rm card} (A)$. By Hermite interpolation we can write
$ f = \sum_{\alpha \in A} f_{\alpha}l_{\alpha}^{-1}   \prod_{\beta 
\in A}  l_{\beta}$,
$ g = \sum_{\alpha \in A} g_{\alpha}l_{\alpha}^{-1}   \prod_{\beta 
\in A}  l_{\beta}$, so that finally, up to a
holomorphic 1-form,
$$ \omega =   \sum_{\alpha \in A} \frac{f_{\alpha} dx + 
g_{\alpha}dy}{l_{\alpha}}.$$  The condition that
$\omega$
  restricted to $\CC^2 \setminus \{0\}$ yields a section of
$\Omega_{\CC^2}^1 ((d\log l_{\alpha})_{\alpha \in A}) $ implies that 
$ f_{\alpha} = -  c_{\alpha} g_{\alpha}$.

Whence, finally,  modulo holomorphic 1-forms, we can write
$\omega = \sum_{\alpha \in A} g_{\alpha} d \log  l_{\alpha}  $.

To prove the last statement, pull back such a 1-form $\omega$:
$p^* \omega = \sum_{\alpha \in A} p^* (g_{\alpha} d \log  l_{\alpha})=
  (\sum_{\alpha \in A} g_{\alpha})  d \log x  + \sum_{\alpha \in A} 
g_{\alpha} d \log (t - c_{\alpha})$.

This form lies in $\Omega^1_{Z}((\log D_{\alpha})_{\alpha \in A})$ if 
and only if
$(\sum_{\alpha \in A} g_{\alpha}(0)) = 0$.

\end{proof}

\begin{cor}\label{dimlog} The dimension of the space $H^0 
(\Omega_{\tilde{Y}}^1 (\log D_i) (E_i - E_{i+2})  )$
is equal to
\begin{itemize}
\item
$2$ in the case $K^2_S = 6$,
\item
$1$ in the case $K^2_S = 5$,
\item
$0$ in the non nodal case $K^2_S = 4$,
\item
$0, 1$ in the nodal case $K^2_S = 4$,according to  $i \neq 1, i=1.$
\end{itemize}

\end{cor}

\begin{proof} The previous lemma shows that, since  $ D_i = D_{i,1} + 
D_{i,2} + D_{i,3} + E_{i+2}$, which by the
way consists of four disjoint curves, then $H^0 (\Omega_{\tilde{Y}}^1 
(\log D_i) (E_i - E_{i+2})  )$ maps onto
$H^0 (\Omega_{\PP^2}^1 (\log D'_{i,1} , \log D'_{i,2} , \log 
D'_{i,3})   )$, where $D'_{i,j}$ is the line image of
the curve $D_{i,j}$.

By the residue exact sequence
$$H^0 (\Omega_{\PP^2}^1 (\log D'_{i,1} , \log D'_{i,2} , \log 
D'_{i,3})   ) = \{ (c_j) \in \CC^3|
\sum_j c_j = 0 \} \cong \CC^2.$$
By 3. we get the subspace  of $\{ (c_j) \in \CC^3|
\sum_j c_j = 0 \} $ such that $c_j = 0$ iff $D'_{i,j}$ contains $P_4$ 
or $P_5$. The rest is a trivial verification.

\end{proof}

\begin{lemma}\label{chi} i) $\chi(\hol_{D_i}(E_i - E_{i+2})) = 8$,

ii) $\chi(\Omega^1_{\tilde{Y}}(E_i - E_{i+2})) = - e(\tilde{Y}) = 
K^2_{\tilde{Y}} - 12$.

\noindent In particular, it follows that
$\chi(\Omega^1_{\tilde{Y}}(\log D_i)(E_i - E_{i+2})) = 8 - 
e(\tilde{Y})= K^2_{\tilde{Y}} -4$.

\end{lemma}

\begin{proof}

The third assertion follows from the first two in view of the exact 
sequence  of locally free sheaves on
$\tilde{Y}$:
$$ 0 \ra \Omega^1_{\tilde{Y}}(E_i - E_{i+2}) \ra 
\Omega^1_{\tilde{Y}}(\log D_i)  (E_i - E_{i+2}) \ra
\hol_{D_i}(E_i - E_{i+2}) \ra 0.
$$

i) Observe that for $1 \leq i,j \leq 3$, we have $(E_i-E_{i+2}) \cdot 
D_{i,j} = 1
  = (E_i-E_{i+2}) \cdot E_{i+2}$, whence $\chi(\hol_{D_i}(E_i - 
E_{i+2})) = 4 \cdot
\chi(\hol_{\PP^1}(1)) = 8$.

ii) In order to calculate $\chi(\Omega^1_{\tilde{Y}}(E_i - E_{i+2}))$ 
we use the splitting principle and
write formally
  $\Omega^1_{\tilde{Y}} = \hol_{\tilde{Y}}(A_1) \oplus 
\hol_{\tilde{Y}}(A_2)$, where $A_1$, $A_2$ are
`divisors'  such that $A_1 + A_2 \equiv K_{\tilde{Y}}$, $A_1 \cdot 
A_2 = e(\tilde{Y})=
12 - K^2_{\tilde{Y}}$. Using that
$(E_i - E_{i+2})^2 = -2$, $K_{\tilde{Y}} \cdot (E_i - E_{i+2}) = 0$, we obtain
\begin{multline*}
  \chi(\Omega^1_{\tilde{Y}}(E_i - E_{i+2})) = 
\chi(\hol_{\tilde{Y}}(A_1 + E_i - E_{i+2})) +
  \chi(\hol_{\tilde{Y}}(A_2 + E_i - E_{i+2})) = \\ = 2 + \frac 12 
((A_1 + E_i - E_{i+2})(E_i - E_{i+2} - A_2) + \\
  + (A_2 + E_i - E_{i+2})(E_i - E_{i+2} - A_1)) = \\ =2 + \frac 12 (-2 
-2 - 2A_1 \cdot A_2) = -e(\tilde{Y}).
\end{multline*}

\end{proof}

\begin{cor}\label{locmod}
  Let $S$ be the minimal model of a  Burniat surface.

Then the dimensions of  the  eigenspaces  of the cohomology groups of 
the tangent sheaf $
\Theta_S$ (for the natural $(\ZZ/ 2\ZZ)^2$-action) are as follows.

\begin{itemize}
\item
$K^2 = 6$: $h^1(S, \Theta_S)^{\inv} = 4$, $h^2(S, \Theta_S)^{\inv} =0$,

$h^1(S, \Theta_S)^i = 0$, $h^2(S, \Theta_S)^i= 2$, for $i \in \{1,2,3 \}$;
  \item
$K^2 = 5$: $h^1(S, \Theta_S)^{\inv} = 3$, $h^2(S, \Theta_S)^{\inv} =0$,

$h^1(S, \Theta_S)^i = 0$, $h^2(S, \Theta_S)^i= 1$, for $i \in \{1,2,3 \}$;

  \item
$K^2 = 4$ of non nodal type: $h^1(S, \Theta_S)^{\inv} = 2$, $h^2(S, 
\Theta_S)^{\inv} = 0$,

  $h^1(S, \Theta_S)^i = h^2(S, \Theta_S)^i = 0$, for $i \in \{1,2,3 \}$.

  \item
$K^2 = 4$ of nodal type:  $h^1(S, \Theta_S)^{\inv} = 2$, $h^2(S, 
\Theta_S)^{\inv} = 0$,

$h^1(S, \Theta_S)^1 = 1 = h^2(S, \Theta_S)^1$,
  $h^j(S, \Theta_S)^i  = 0$, for $i \in \{2,3 \}$.

\end{itemize}
\end{cor}
\begin{proof}
It is a straightforward consequence of corollary \ref{dimlog}, of 
lemma \ref{chi}, and of the
Enriques-Kuranishi formula
$ \chi ( \Theta_S) = - 10 \chi(\hol_S) + 2 K^2_S$.

\end{proof}

We are in the position to state the main results of this section.

Recall that for surfaces of general type we have two moduli spaces:
one is the moduli space
$\mathfrak M^{min}_{\chi, K^2}$ for minimal models $S$ having
$\chi(\hol_S) = \chi$, $K^2_S = K^2$, the other is the moduli space
$\mathfrak M^{can}_{\chi, K^2}$ for canonical models $X$ having
$\chi(\hol_X) = \chi$, $K^2_X = K^2$; the latter is called the Gieseker
moduli space. Both are quasi projective schemes by Gieseker's theorem
(\cite{gieseker})  and by the fact that there is a natural morphism
$\mathfrak M^{min}_{\chi, K^2} \ra \mathfrak M^{can}_{\chi, K^2}$ which is a
bijection. Their local structure as
complex analytic spaces is the quotient of the base of the Kuranishi
family by the action
of the finite group $ \Aut(S) = \Aut(X)$. Usually the scheme
structure of $\mathfrak M^{min}_{\chi, K^2}$ tends to be more singular
than the one of
$\mathfrak M^{can}_{\chi, K^2}$ (see e.g. \cite{enr}).

Recall moreover that in the following theorem $K_S$ is always ample, 
thus the minimal
and canonical model coincide. Instead, later on, for surfaces with 
$K^2 = 4$ of  nodal type,
$S$ contains exactly  one  -2 curve $E$, thus the canonical model $X$ has
always exactly one singular point, an $A_1$-singularity.

\begin{theo}\label{open} The three respective subsets  of the  moduli 
spaces of minimal
surfaces of general type $\mathfrak M^{min}_{1, K^2}$   corresponding to
Burniat surfaces with $K^2 = 6$, resp. with $K^2 = 5$, resp.
Burniat surfaces with $K^2 = 4$ of non nodal type,
are irreducible open sets, normal, unirational
  of respective dimensions 4,3,2.

More precisely, the base of the Kuranishi family of $S$ is smooth.
\end{theo}

\begin{proof}
By corollary \ref{locmod} the tangent space to the Kuranishi family of $S$,
$H^1 (\Theta_S )$, consists of invariants for the action of the group 
$G := (\ZZ / 2 \ZZ)^2$.

It follows then (cf. \cite{cime88} lecture three, page 23) that all 
the local deformations admit the $G$-action,
hence they are bidouble covers of a deformation of the Del Pezzo 
surface $\tilde{Y}$.

Moreover, the dimension of $H^1 (\Theta_S )$ coincides with the 
dimension of the image
of the Burniat family containing $S$ in the moduli space $\mathfrak 
M^{min}_{1, K^2_S}$,
hence the Kuranishi family of $S$ is smooth, and coincides with the 
Burniat family by the
above argument.

Alternatively, one could show directly that the Kodaira Spencer map 
is bijective,
or simply observe that  a finite morphism
between smooth manifolds of the same dimension is open.

   Observe then that the  quotient of a smooth variety
by the action of a finite group is normal.

Finally, the Burniat
family is parametrized by a (smooth) rational variety.

\end{proof}

We shall prove in the final section that these irreducible components 
are not only unirational,
but indeed rational.

\begin{theo}\label{mainnodal}

The subset  of the Gieseker moduli space
$\mathfrak M^{can}_{1, 4}$ of canonical
surfaces of general type  $X$  corresponding to
Burniat surfaces $S$ with $K_S^2 = 4$ and of  nodal type
is an irreducible  open set  of dimension 2, unirational and
  everywhere non reduced.

More precisely, the base of the Kuranishi family of $X$ is locally 
analytically isomorphic
to $\CC^2 \times Spec (\CC[t]/ (t^m))$, where $m$ is a fixed integer, 
$m \geq 2$.

The corresponding subset  of the  moduli space $\mathfrak M^{min}_{1, 
4}$ of minimal
surfaces $S$ of general type is also   everywhere non reduced.

More precisely, the base of the Kuranishi family of $S$ is locally 
analytically isomorphic
to $\CC^2 \times Spec (\CC[t]/ (t^{2m}))$.
\end{theo}

The rest of this section is devoted to the proof of theorem 
\ref{mainnodal}, the first argument,
as in theorem \ref{open}, being that all the local deformations admit 
the $G$-action.
This is more difficult to show since $H^1 (\Theta_S )$ consists of 
the direct sum
of the space of $G$-invariants, which has dimension 2, and of a 
1-dimensional character space.

Let $S$ be the minimal model of a Burniat surface with $K^2 = 4$ of nodal type,
  let  $X$ be its canonical
model, and   denote by $\pi \colon S \ra X$ the blow down of the 
unique $(-2)$-curve $E$ of $S$
  (lying over
$D_{1,2}$).

By \cite{burnswahl} we know that $\pi_* \Theta_S = \Theta_X$ and that 
$H^1_E(\Theta_S)$
has dimension 1. It follows then
from the Leray spectral sequence for $\pi_*$:

\begin{itemize}
  \item[-] $H^2(S, \Theta_S) = H^2(X, \Theta_X)$,
\item[-] $H^1(S, \Theta_S) = H^1(X, \Theta_X) \oplus \mathcal{R}^1 
\pi_* \Theta_S = H^1(\Theta_X) \oplus
H^1_E(\Theta_S)$.
\end{itemize}

In particular $h^1(\Theta_X) = 2$.

On the other hand, by the local to global Ext-spectral sequence, we 
have the ``five term exact sequence'':
\begin{multline}\label{fiveterm}
  0 \ra H^1(X, \Theta_X) \ra \Ext^1_{\hol_X}(\Omega^1_X, \hol_X) \ra \\
  \ra H^0(X, \mathcal{E}xt^1_{\hol_X}(\Omega^1_X, \hol_X))
\ra H^2(X, \Theta_X) \ra \Ext^2_{\hol_X}(\Omega^1_X, \hol_X) \ra 0.
\end{multline}

Note that the above exact sequence is a $G = (\ZZ / 
2\ZZ)^2$-equivariant sequence of $\CC$-vector spaces, since all
sheaves have a natural $G$-linearization.
We proceed now to calculate the decomposition in character spaces of
the single terms of the exact sequence.

\begin{lemma}\label{h0ext}
  The 1-dimensional space $H^0(X, \mathcal{E}xt^1_{\hol_X}(\Omega^1_X, 
\hol_X))$ is a space of invariants for
the $G$-action.
\end{lemma}

\begin{proof} Recall that $D_{1,2}$ is a $(-2)$-curve on $\tilde{Y}$ 
and  $X$ is a bidouble cover of the
nodal Del Pezzo surface $Y'$ of degree $4$ obtained contracting $D_{1,2}$.

Moreover, the curve $D_{1,2}$ intersects exactly two other 
irreducible components of
the branch locus, namely, $D_{2,1}$ and $E_1$, which is also a 
component of $D_2$.

We want to describe the structure of
the morphism
$f
\colon X \ra Y'$ locally around the $A_1$-singular point $P'$.

Locally around $P'$ we can assume that  $Y' = \{z^2 = xy \}$.

Then, locally around the
node  $P$ of $X$,  $X = \{w^2 = uv \}$, and the bidouble covering $f 
\colon X \ra Y'$ is given by the
equations:
$w^2 = z$, $u^2 = x$, $v^2 = y$.

In fact, the intermediate double cover branched only on $D_{1,2}$ 
corresponds to the
double cover branched only on $P'$, and given by $\Phi : \CC^2 \ra Y'$,
such that  $\Phi (u,v) = (u^2, v^2, uv) : = (x,y,z) $,
while $X$ is the double cover  $ w^2 = uv$ branched on the inverse 
images of the lines
$x=z=0  $ and $ y=z=0$ (observe that for $A_1$ the two $G$ actions 
listed in Table 3 of
\cite{autRDP}, page 93 ares conjugate to each other).

The local deformation of the $A_1$-singularity on $X$ is given by
$$ X_t=\{w^2 = uv + t\}.
$$ Then $t \in \CC$ is a trivial representation of $G$ and therefore $X_t$
yields a family of  $G$-coverings
of $Y'$ described
by the equations $w^2 = z+t$, $u^2 = x$, $v^2 = y$.  This proves the claim.

\end{proof}

\begin{cor}
  The obstruction map
$$
\ob \colon  H^0(X, \mathcal{E}xt^1_{\hol_X}(\Omega^1_X, \hol_X))
\ra H^2(X, \Theta_X)
$$ is identically zero.
\end{cor}

\begin{proof}
  Recall that ``$\ob$'' is a $G$-equivariant homomorphism. Since 
\linebreak[4] $H^0(X,
\mathcal{E}xt^1_{\hol_X}(\Omega^1_X, \hol_X))$ is a trivial 
$G$-representation, while $H^2(X,
\Theta_X)=H^2(S, \Theta_S)$ is a nontrivial $G$-representation (cf. 
corollary \ref{locmod}), it follows that
$\ob = 0$.

\end{proof}

\begin{cor}
  $\Ext^2_{\hol_X}(\Omega^1_X, \hol_X) = \Ext^2_{\hol_X}(\Omega^1_X, \hol_X)^1$.
\end{cor}

\begin{proof}
  Follows immediately from the exact sequence (\ref{fiveterm}) and the 
above corollary.
\end{proof}

\begin{lemma}
$H^0(X,\mathcal{R}^1 \pi_* \Theta_S) =  H^1_E(\Theta_S)$ is a non 
trivial character of $G$.
\end{lemma}

\begin{proof}
  By the theorem of Brieskorn-Tjurina (\cite{brieskorn1}, 
\cite{tjurina}), the simultaneous resolution of the node
on $X$ is given by $\frac{w - \tau}{u} = \frac{v}{w+\tau}$ where one 
has made the base change $\tau^2 = t$,
using the notation of
the proof of lemma
\ref{h0ext}.  The action of $G$ lifts in a unique way to the 
simultaneous resolution of the family
since $\tau$ must be an eigenvector with character equal to the same 
character of $w$
(observe that both $w- \tau , w+
\tau$ are eigenvectors).

Since $\CC \tau \cong
H^1_E(\Theta_S)$ as $G$-representation, we have proven that
$H^1_E(\Theta_S)$ is an eigenspace corresponding to a  non trivial 
character of $G$.

\end{proof}

Since $H^1(S, \Theta_S) = H^1(\Theta_X) \oplus H^1_E(\Theta_S)$, the 
above lemma and corollary
\ref{locmod} immediately imply the following

\begin{cor}
  $H^1(X, \Theta_X)$ is a trivial $G$-representation, hence also 
$\Ext^2_{\hol_X}(\Omega^1_X, \hol_X)$.
\end{cor}

Now we are ready to prove the following
\begin{prop}
  Let $X$ be the canonical model of a Burniat surface with $K_S^2 = 4$ 
of nodal type. Then all deformations of
$X$ are deformations of the pair $(X,G)$.
\end{prop}

\begin{proof}
  Since, by the above considerations, $G$ acts trivially on the base 
of the Kuranishi family of $X$, it follows that
$\Def(X) = \Def(X)^G$.

\end{proof}

The consequence is then that also all deformations of $S$ are 
deformations of the pair $(S,G)$.

The main theorem of this section (thm. \ref{mainnodal}) will now 
follow once we have proven

\begin{prop}\label{open4}
  Burniat surfaces with $K_S^2 =4$ of nodal type yield an open set in 
the moduli space.
\end{prop}

\begin{proof} Let $S$ be the minimal model of a Burniat surface with 
$K_S^2 =4$ of nodal type. Then $S/G =
\tilde{Y}$, where $\tilde{Y}$ is a weak Del Pezzo surface.

Now, by the above, any small deformation $S_t$ of $S$ is in fact a 
deformation of $(S,G)$.
It suffices to show that
$S_t/G$ is again a weak Del Pezzo surface, i.e., the $(-2)$-curve 
remains under a small deformation.

We remark that  the $(-2)$-curve on $\tilde{Y}$ is $D_{1,2}$,
which  is a connected component of $D_2$, hence $E$
  is a connected component of the fixed point set $ Fix (\sigma_2)$ of 
an element $\sigma_2 \in G$.

Let now $\sS \ra T$ be a one parameter family of minimal models, such 
that $G$ acts on $\sS \ra T$,
with trivial action on $T$ and the given action on the central fibre.
Then the component of $ Fix (\sigma_2)$ in $\sS$ has dimension 2, 
whence all the deformations
$S_t$ of $S$ carry a -2 curve $E_t$ deformation of $E$.
It follows that the quotient of $E_t$ yields a -2 curve on
$\tilde{Y}$.

In other words, we have shown that all deformations of $X$ are 
equisingular, therefore
$ Def(X) \subset H^1 (\Theta_X)$. The Burniat family  shows that $ 
dim ( Def(X)) \geq 2$,
whence  set
theoretically $ Def(X) =  H^1 (\Theta_X)$.

Choosing coordinates $(t_1, t_2, t_3)$ for 
$\Ext^1_{\hol_X}(\Omega^1_X, \hol_X) $
such that $\{t_3=0\}$ is the hyperplane $ H^1 (\Theta_X)$, we see that 
the Kuranishi equation
is a power of $t_3$, say $t_3^m$. Since the Kuranishi equation has 
differential vanishing at the origin,
it follows that $ m \geq 2$.

Now, the local map  $ H^1(\Theta_S) \ra \Ext^1_{\hol_X}(\Omega^1_X, 
\hol_X) $ (cf.  theorem
2.6 of \cite{burnswahl}, see also \cite{enr})
is given by $(s_1, s_2, s_3) \mapsto (s_1, s_2, s_3^2)$, and $ 
Def(S)$ is the pull back of
$Def (X)$. Hence $ Def(S)$ is the subscheme $ s_3^{2m} = 0$.

\end{proof}

\section{One parameter limits of secondary Burniat 
surfaces}\label{degenerations}

In this section we shall show that Burniat surfaces with $ K^2 \geq 
4$ form a closed set of
the moduli space.

This will be accomplished through the
  study of limits of one parameter families of such Burniat surfaces.

We get a new result only in the case of secondary Burniat surfaces 
with $4 \leq K^2 \leq 5$.
The argument is exactly
the same for  $ K^2 = 6$, but in this case we are just giving a 
fourth proof after
the ones given in \cite{mlp}, in \cite{alpar} and in part I (\cite{burniat1}).

Let $Y'$ be a normal $\QQ$- Gorenstein surface and denote the 
dualizing sheaf of $Y'$ by $\omega_{Y'}$.

Then there is a minimal natural number $m$ such that 
$\omega_{Y'}^{\otimes m}$ is an invertible sheaf and  it
makes sense to define $\omega_{Y'}$ to be ample, respectively 
anti-ample; $Y'$ is Gorenstein iff $ m=1$.

We shall need the following
\begin{prop}\label{gorenstein}
  Let $Y'$ be a normal $\QQ$-Gorenstein Del Pezzo surface (i.e., 
$\omega_{Y'}$ is anti-ample) with
  $K^2_{Y'} \geq 4$.
Then $Y'$ is in fact Gorenstein.
\end{prop}

\begin{proof}
  Assume that $ m \geq 2$.  Then (cf. \cite{ypg}, Proposition on page
362),  there is a $\ZZ / m \ZZ$-Galois covering $p \colon W \ra Y'$ 
such that $W$ is Gorenstein and such that
$K_W = p^* K_{Y'}$, where $\omega_{Y'}$ is the sheaf associated to 
the Weil divisor $K_{Y'}$.  $p$ is only
branched on the singular points of
$Y'$ which are not Gorenstein.

Since $\omega_{Y'}$ is anti ample, it follows that $K_W$ is anti 
ample, hence $W$ is a normal Gorenstein Del
Pezzo surface.  As it is well known (cf.  e.g. \cite{catmangolte}, 
Theorem 4.3) $W$ is smoothable and
  in particular $K_W^2 \leq 9$, indeed $K_W^2 \leq 8$ if $W$ is singular.

On the other hand: $K_W^2 = mK_{Y'}^2 \geq 4m$ and this implies that 
$m=2$, $K_{Y'}^2 = 4$.

Therefore $K_W^2 = 8$, whence either $W$ is the blow up of the plane 
in one point, or $W=Q$ a quadric  in
$\PP^3$.

  If $W$ is smooth then $Y' = W/(\ZZ/ 2 \ZZ)$ has only 
$A_1$-singularities and  is Gorenstein.

It remains therefore to exclude the case that $W$ is the quadric cone.

  In this case $Y' = Q/i$, where $i$ is an involution on $Q$: since 
the quotient is not Gorenstein (see
\cite{autRDP},Table 2 and Theorem 2.2, page 90)
  it acts on  the tangent space at the node of $Q$ as $- \Id$.

  The involution $i$  on $Q$ acts linearly on the anticanonical model 
of $Q$, thus $i$ extends to a
linear  involution
$I$ on $\PP^3$.

  The vertex $v \in Q$ is an isolated fixed point of $I$, and $I$ acts 
as $- \Id$ on the tangent space of $v$.
Therefore $H^0(Q, \hol_Q(1))$ splits into two  eigenspaces of 
respective dimensions $3,1$.

In particular there is a pointwise fixed   hyperplane $H \subset 
\PP^3$ for $I$. Since then $C:= Q \cap H$ is
pointwise fixed by $I$, we contradict the fact that $I$ has only 
isolated fixed points on $Q$.

This implies that $Y'$ is Gorenstein.

\end{proof}

\begin{prop}\label{familydc}
  Let $T$ be a smooth affine curve, $t_0 \in T$, and let $f \colon 
\mathcal{X} \ra T$ be a flat  family of
canonical surfaces. Suppose that $\X_t$ is the canonical model of a 
Burniat surface with $4
\leq K_{\mathcal{X}_t}^2 $ for $t \neq t_0 \in T$. Then there is an
action of  $G: = (\ZZ / 2\ZZ)^2$ on $\X$ yielding a one parameter 
family of finite $(\ZZ /
2 \ZZ)^2$-covers
\begin{equation*}
\xymatrix{
  \mathcal{X}\ar[dr]_f\ar[rr]&& \mathcal{Y} \ar[dl] \\ & T&,& }
\end{equation*} (i.e., $\X_t \ra \mathcal{Y}_t$ is a finite $(\ZZ / 2 
\ZZ)^2$-cover), such that
$\mathcal{Y}_t$ is a Gorenstein Del Pezzo surface for each $t \in T$.

\end{prop}

\begin{proof}
  Note that $\mathcal{X}$ is Gorenstein, since $T$ is smooth and the fibres have 
hypersurface singularities.

Since $\mathcal{X} \setminus f^{-1}(t_0) \ra T \setminus \{t_0\}$ is 
a family of canonical models of
Burniat surfaces, we have a $(\ZZ / 2 \ZZ)^2$-action on $\mathcal{X} 
\setminus f^{-1}(t_0)$
(this is the Galois group action for the bicanonical map).

By
\cite{catravello}, thm. 1.8, the $(\ZZ / 2 \ZZ)^2$-action extends to 
$\mathcal{X}$.

We set $\mathcal{Y} := \mathcal{X} / (\ZZ / 2 \ZZ)^2$ and we denote 
by $\Phi$ the finite morphism
$\mathcal{X} \ra \mathcal{Y}$.

We have for all $t \in T$:
\begin{itemize}
\item $K_{\mathcal{Y}_t} = K_{\mathcal{Y}} | _{\mathcal{Y}_t}$;
\item $K_{\mathcal{X}_t} = K_{\mathcal{X}} | _{\mathcal{X}_t}$.
\end{itemize}

Moreover,
$$
2 K_{\mathcal{X}} = 2 \Phi^*(K_{\mathcal{Y}}) + \sB,
$$
where $\sB$ is the branch divisor of  $\Phi \colon \mathcal{X} \ra 
\mathcal{Y}$.

Since for $t \neq t_0$ we have
$2 K_{\mathcal{X}_t} = -\Phi^*(K_{\mathcal{Y}_t})$, it follows that
$$2 K_{\mathcal{X}}
+ \Phi^*(K_{\mathcal{Y}}) \equiv 0 \ \ {\rm on} \ \  \mathcal{X} 
\setminus \mathcal{X}_{t_0}.
$$

Since however $\mathcal{X}_{t_0} = f^{-1}(t_0)$ is irreducible, we 
obtain (after possibly restricting $T$) that $2
K_{\mathcal{X}} + \Phi^*(K_{\mathcal{Y}}) \equiv 0$ on $\mathcal{X}$.

In particular, $2K_{\mathcal{X}_t} =
-\Phi^*(K_{\mathcal{Y}_t})$ for all $t \in T$, which implies that 
$-K_{\mathcal{Y}_t}$ is ample for all $t \in T$.

Moreover,  $K_{\mathcal{X}_t}^2 =
K_{\mathcal{Y}_t}^2$ for all $t \in T$.

By construction,
$\mathcal{Y}_t$ is a Gorenstein Del Pezzo surface for $t \neq t_0$, 
and $\mathcal{Y}_{t_0}$ is a normal
$\QQ$-Gorenstein Del Pezzo surface, whence it is Gorenstein by prop. 
\ref{gorenstein}.

\end{proof}

This implies immediately the following:

\begin{cor}
  Consider a one parameter family of bidouble covers $\mathcal{X} \ra 
\mathcal{Y}$ as  in prop. \ref{familydc}.
Then the branch locus of $\mathcal{X}_{t_0} \ra \mathcal{Y}_{t_0}$ is 
the limit of the branch locus of
$\mathcal{X}_t
\ra \mathcal{Y}_t$, and it is reduced.
\end{cor}

Note that the limit of a line on the del Pezzo surfaces 
$\mathcal{Y}_t$ is a line on the del Pezzo
surface $\mathcal{Y}_{t_0}$, and, as a consequence of the above assertion,
  two lines in the branch locus in
$\mathcal{Y}_t$ cannot tend to the same line in
$\mathcal{Y}_{t_0}$.

\begin{rem} Let $X$ be the canonical model of a Burniat surface with 
$4 \leq K_X^2 \leq 6$.
Recall once more that $X$ is smooth
for $K_X^2 = 6,5$, and for $K_X^2 = 4$ in the non nodal case. For 
$K_X^2 = 4$  and the nodal case,
$X$ has one
ordinary node.

In all three cases the branch locus consists of the union of
$3$ hyperplane sections, containing   $\nu$
lines and $\frac{1}{2} (3 K^2_X - \nu ) $ conics , where
\begin{itemize}
\item[a)] $\nu = 6$ for $K_X^2 = 6$,
\item[b)] $\nu = 9$ for $K_X^2 = 5$,
\item[c)] $\nu = 12$ for $K_X^2 = 4$ non nodal,
\item[d)] $\nu = 10$ for $K_X^2 = 4$ nodal.
\end{itemize}
In fact, in case a) the 6 lines contained in the branch locus are: 
$D_{i,1}$, $1 \leq i \leq3$,
  $E_1$, $E_2$, $E_3$. In case b) the 9
lines contained in the branch locus are: $D_{i,j}$, $1 \leq i \leq3$, 
$1 \leq j \leq 2$, $E_1$, $E_2$, $E_3$.

In case c) the 12 lines in the branch locus of the bidouble cover 
are: $D_{i,j}$, $1 \leq i,j \leq3$, $E_1$, $E_2$,
$E_3$, and finally in case d) the 10 lines are: $D_{i,1}$, $1 \leq i 
\leq3$,  $D_{2,2}$, $D_{2,3}$,
$D_{3,2}$, $D_{3,3}$, $E_1$, $E_2$, $E_3$.
\end{rem}

We shall use the following:

\begin{prop}\label{wdp}[\cite{bucarest}, prop.  1.7.] A weak Del 
Pezzo surface $W$ is either
\begin{itemize}
\item[-] $\PP^1 \times \PP^1$, or
\item[-] $\FF_2$, or
\item[-] the blow up $\hat{\PP}^2 (P_1, \ldots , P_r)$, $r \leq 8$,
\end{itemize}
at $r$ distinct points  $P_1, \ldots , P_r$ satisfying  the following 
three conditions:
\begin{itemize}
\item[i)] no more than $3$ $P_i$'s are collinear;
\item[ii)] no more than $6$ $P_i$'s lie on a conic;
\item[iii)] the set $\{P_1, \ldots , P_r \}$ can be partitioned into 
subsets $\{P_{i_1}, \ldots , P_{i_k} \}$ with
$ P_{i_1} \in \PP^2$, $P_{i_{(j+1)}}$ infinitely near to $P_{i_j}$, 
but not lying on the proper transform of
$P_{i_{(j-1)}}$.
\end{itemize}
\end{prop}
   Since  weak Del Pezzo surfaces $W$ are exactly the minimal resolutions
of singularities of normal Gorenstein Del Pezzo surfaces $Z$, we use 
the above result to show the following
technical, possibly well known result:
\begin{prop}\label{lines}
Let $Z$ be a  normal Gorenstein Del Pezzo surface of degree $d$.

Then $Z$ contains no line for $d=9,8$ unless $ Z = \F_1$, which 
contains one line.

For $d=7$ $Z$ contains 2 or 3 lines, and is smooth in the latter case.

If $d= 6,5,4$  $Z$ contains at most $6$, respectively $10$, 
respectively $16$ lines. If
  $Z$ contains at least
$6$, respectively $9$, respectively $13$ lines (i.e., irreducible 
curves $C$ with $ C \cdot K_Z = -1$),
then $Z$ is smooth.

  Assume that $d=4$ and that $Z$ contains at least 10 lines. Then we 
have the following possibilities:
\begin{itemize}
\item[i)] $Z$ is smooth and has $16$ lines;
\item[ii)] $Z$ has exactly  one singular point,
of type $A_1$,  $Z$ has $12$ lines and
$Z$ is the anticanonical model of the weak Del Pezzo surface obtained
blowing up the plane in 5 distinct points such that  three of them 
are collinear.
\end{itemize}
\end{prop}

\begin{proof}
If $W$ is $\F_0 = \PP^1 \times \PP^1$, $\F_2$ or $\PP^2$, then 
obviously $W$ contains no line.

Thus we may assume that $W$ is the blow up of the plane at $P_1, 
\dots P_r$, with
$ r = 9 - d$. For $r=1$ there is only the line $E_1$, where we denote 
as customary
by $E_i$ the full transform of the point $P_i$.

Any line $C$ is in particular an effective divisor such that $C^2 = C 
K_W = -1$, and
in particular  it is contained in some anticanonical divisor $ H = 3L 
- \sum_j E_j$,
  where $L$ is the nef and big divisor
pull back of a line of
$\PP^2$.

Thus $ C \equiv a L - \sum_j b_j E_j$ and since  $ L C \geq 0$, $ L ( 
H-C) \geq 0$,
one gets $ 0 \leq a \leq 3$.

As usual $C^2 = C K_W = -1$ implies
$$  a^2 + 1 = \sum_j b_j ^2, \  \sum_j b_j = 3 a - 1 \Rightarrow 
\sum_j b_j (b_j -1) = (a-1)(a-2). $$

The right hand side vanishes for $a=1,2$ and equals $2$ for $a=0,3$ 
while each summand on the left
side of the last equality is at least 2 unless $b_j=0$ or $b_j =1$.

  Not considering the $b_j$'s equal to zero,
  for $a=0$ one has one  $b_j= -1$, for $a=1$ one
has two
$b_j= 1$, for $a=2$ one has five
$b_j= 1$.

While $ a=3$ can only  occur for $r \geq 7$, with  one $b_j$ equal to 
$2$, and six
equal to $1$.

This gives the a priori bound that the number of lines is at most 
$$ 
N(r): = r + \begin{pmatrix} r \\ 2 \end{pmatrix} + \begin{pmatrix} r \\ 5 \end{pmatrix}.
$$

This gives the number of lines in the case where $- K_W$ is ample, namely,
for $ d = 7,6,5,4$ we get $r= 2,3,4,5$ and $ N =  3, 6, 10, 16$.

Since, if $- K_W$ is ample, each such divisor  is linearly equivalent 
to an unique
  effective one which is irreducible.

If $- K_W$ is not ample but nef, then there are -2 curves $D$, i.e., 
irreducible divisors $D$
with $ D \equiv a L - \sum_j b_j E_j$, $ 0 \leq a \leq 3$, and
  $D^2 = -2, D K_W = 0$. These conditions  are equivalent to
$$  a^2 + 2 = \sum_j b_j ^2, \  \sum_j b_j = 3 a  \Rightarrow  \sum_j 
b_j (b_j -1) = (a-1)(a-2). $$

By the same token $a=1,2$ implies $b_j=1,0$ and we get for $a=1$ 
three $b_j=1$, for $a=2$
six $b_j=1$.  For $a=0$ we get a divisor of the form $E_i - E_j$, for 
$a=3$ must be $r \geq 8$
and one $b_j=2$, seven $b_j=1$.

What is left is to show that each -2 curve $D$ makes the number of 
lines diminish
sufficiently.

For $a=2$, we must have $ r \geq 6$ (and then we lose 6 lines); for 
$a=1$, $ D = L - E_i-E_j - E_k$,
  we lose 3 lines, since
$L- E_i -E_j  = D + E_k$. We also lose, if $ r \geq 5$, $C(r-3,2)$ 
lines of the form $D + ( L-E_h - E_l)$.

Since we assume $r \leq 7$, let us see what happens
if  $ D =  E_i-E_j $ is effective. This means that $P_j$ is 
infinitely near to $P_i$,
so we have a string of infinitely near points as in iii) of 
proposition \ref{wdp}.

Assume that this string is  $ P_{i_1}, \ldots , P_{i_k} $. Then each $E_{i_h}$ 
is not
irreducible for $ h=1, \dots , k-1$. Also the effective divisor $ L - 
E_j -  E_{i_h}$
is not
irreducible for $ h=2, \dots , k$, and for $P_j$ not infinitely near 
to $ P_{i_1}$.
Moreover   $ L - E_{i_1} -  E_{i_2}$ is effective, and contained in
$ L - E_{i_h} -  E_{i_l}$ whenever $ h \leq l$ are not equal to $1,2$.

The loss is therefore at least
\begin{multline*}
(k-1)+ (k-1)  (r-k) + \frac{1}{2} (k+1)(k-2)= \\
= (k-1) [r-(k-1)] +  \frac{1}{2} (k+1)(k-2).
\end{multline*}

For $k =  2$ we get a loss of $r-1$ lines, otherwise a bigger loss.

We want to finally show that the case $ r = 5$ and $k=2$ yields the 
same surface
which is encountered for   $ r = 5$, no infinitely near points, but 3 
collinear points.

Consider then, as in the nodal case, 5 points such that $P_1, P_4, 
P_5$ are collinear, and
let $\Psi \colon \PP^2 \dasharrow \PP^2$ be the birational standard 
Cremona transformation
based on the points $P_1, P_2, P_5$. On the Del Pezzo $\tilde{Y}$ 
obtained blowing up the 5 points
$\Psi$ corresponds to the linear system $ 2L - E_1 - E_2 - E_5$. This 
system contracts
the -2 curve to a point, as well as the lines $E_4, E_3$, $L -  E_1 - 
E_2 $, $L - E_2 - E_5$.

Since  the -2 curve intersects $E_4$, we get also a representation of 
$\tilde{Y}$
as the blow up of the plane in five points, of which one infinitely 
near to the other.

\end{proof}

\begin{theo}\label{closed} Each family of
Burniat surfaces with $K^2 = 4,5,6$  yields a closed subset
  of the moduli space.
\end{theo}

\begin{proof} Consider a one parameter family of bidouble covers 
$\mathcal{X} \ra \mathcal{Y}$ as  in prop.
\ref{familydc}, such that $\mathcal{X}_t \ra \mathcal{Y}_t$ is the 
bicanonical map of a Burniat
surface  for $ t \neq t_0$.  Then $\mathcal{X}_{t_0} \ra
\mathcal{Y}_{t_0}$ is a bidouble cover of a normal Gorenstein Del 
Pezzo surface of degree $K^2_{\X_t}$ and
$\mathcal{X}_{t_0} $ has canonical singularities.

  Moreover, the branch locus of
  $\mathcal{X}_{t_0} \ra
\mathcal{Y}_{t_0}$  is the limit of the branch locus of 
$\mathcal{X}_t \ra \mathcal{Y}_t$,
hence it contains   at least $3 ( 8 - K^2_{\X_t})$ lines in the non 
nodal case, and $10$
in the nodal case.

Then by proposition  \ref{lines} $\mathcal{Y}_{t_0}$ is smooth for 
$K^2_{\X_t} \geq 5$,
while for $K^2_{\X_t} = 4$ it has at most one node.

Thus, for $K^2_{\X_t} \geq 5$,  $\mathcal{X}_{t_0}$ is again a Burniat surface.

Assume that $K^2_{\X_t} = 4$ and that we are in the non nodal case. 
We are done unless
$\mathcal{Y}_{t_0}$ has a node.

In this case  every line of $\mathcal{Y}_{t_0}$ is a component of the 
branch locus.

Note that through the node of $\mathcal{Y}_{t_0}$ pass 4 lines. By 
\cite{autRDP}, table 3, page 93,
  a bidouble
cover of a node branched in $4$ lines is no longer a rational double 
point, and we have reached a contradiction.

Finally, in the nodal case, we have seen that the family 
$\mathcal{Y}_{t}$ is equisingular.
By proposition \ref{lines} the minimal resolution of 
$\mathcal{Y}_{t_0}$ is the blow up of $\PP^2$
in 5 distinct points, none infinitely near, with $P_1, P_4, P_5$ collinear.

A similar  representation holds for the minimal resolution $W_t$ of 
$\mathcal{Y}_{t}$;
by the above argument two of the  lines  passing through the node 
cannot be part of the
branch locus. Thus the branch locus for each $W_{t}$ consists of the -2 curve,
of 10 (Del Pezzo) lines and
a (Del Pezzo) conic.  Thus the configuration of the branch locus 
remains of the same type
and the central fibre $\mathcal{X}_{t_0}$ is again a nodal Burniat surface.

\end{proof}

\section{Proof of the main theorems and corollaries}

All the statements (except the one concerning rationality) of the two 
main theorems follow combining the two
theorems \ref{open} and
\ref{mainnodal}, showing that the Burniat families for $ K^2_S \geq 
4$ form open sets,
with theorem \ref{closed}, showing that they form closed sets.

There remains to prove the  rationality of the four connected 
components $\sC$ of
the moduli space constituted by Burniat surfaces with $K^2_S \geq 4$.
This is automatical for $K^2_S = 4$ since $\sC$ has dimension 2, and 
by Castelnuovo's criterion
every unirational surface (over $\CC$) is rational.

We deal next with the case $K^2_S = 5$.

\begin{theo}
Let $\sC$  be the connected component  of
the moduli space constituted by Burniat surfaces with $K^2_S = 5$.

Then  $\sC$  is a rational 3-fold.
\end{theo}
\begin{proof}
The bicanonical map of $S$ yields a bidouble cover $ \Phi_2 \colon S 
\ra \tilde{Y}$,
where $\tilde{Y}$ is the Del Pezzo of degree 5 obtained blowing up the plane in
the 4  reference points.

As we saw, the branch locus consists of nine Del Pezzo lines and of 3
Del Pezzo conics. Thus there is exactly
one line which is not contained in the branch locus, and we can 
contract it, obtaining
a Del Pezzo surface $Y$ of degree 6. The branch locus contains now 
the six lines of $Y$.

Let us fix an identification of the Galois group of  $ \Phi_2$ with 
$G = (\ZZ/2 \ZZ)^2$.
Then these 6 lines, which
form an hexagon, are such that each pair of opposite sides is 
labelled by an element
in $G \setminus \{0\}$.

There are two ways to contract three such lines (one for each pair) 
and obtain the projective
plane $\PP^2$, and they are related by the standard Cremona transformation
$ (x_1 : x_2 : x_3)  \mapsto (x_1^{-1} : x_2^{-1} : x_3^{-1})$ 
associated to the linear system
$| 2L - E_1 - E_2 - E_3|$.

We chose the points $P_1, P_2, P_3, P_4$ as the reference points 
($P_4 = (1 : 1 : 1)$),
and we consider now the triples of lines  corresponding to $ D_{i,3}$,
which have necessarily an equation of type $ x_{i+2} = a_i x_{i+1}  $.

The Cremona transformation acts by $a_i \mapsto a_i^{-1}$,  the 
cyclical permutation
of coordinates
cyclically permutes  the three numbers
$a_1, a_2, a_3$, while the transposition   exchanging $1 $ with $2$ sends
$$(a_1, a_2, a_3)  \mapsto (a_2^{-1}, a_1^{-1}, a_3^{-1}).$$ 
Composing the action of  such
a transposition with the action of the Cremona transformation we get 
the transposition of $a_1$ and $a_2$.

We conclude that there is a
subgroup of index two, isomorphic to $\mathfrak S_3$, acting   on 
the three numbers
$a_1, a_2, a_3$ via
the standard permutation action of the symmetric group $\mathfrak S_3$ .

The full group by which we want to divide is generated by this subgroup and
by the
Cremona transformation.
The invariants for the permutation representation of $\mathfrak S_3$ 
are the three elementary symmetric
functions
$\sigma_1, \sigma_2, \sigma_3$. The Cremona transformation acts on the
field  $K$ of  $\mathfrak S_3$- invariants by
$$\sigma_3  \mapsto \sigma_3^{-1}, \ \sigma_1  \mapsto  \sigma_2 \sigma_3^{-1},
\sigma_2  \mapsto  \sigma_1 \sigma_3^{-1}. $$

Obvious invariants are $$ \sigma_1 +  \sigma_2 \sigma_3^{-1} : = y_1,
\sigma_2 +  \sigma_1 \sigma_3^{-1} : = y_2, \sigma_3 +  \sigma_3^{-1} 
: = y_3.$$

Let $F$  be the field $\CC (  y_1,  y_2,  y_3)$: to show that $F$ is 
the whole field of
invariants it will suffice to show that $ [ K : F ] = 2 $.

Now, $ F ( \sigma_3)$ is a quadratic extension of $F$ , and the two 
linear equations
in $ \sigma_2 ,   \sigma_1 $
$$ \sigma_1 +  \sigma_2 \sigma_3^{-1} = y_1, \sigma_2 +  \sigma_1 
\sigma_3^{-1} = y_2$$
have determinant $ 1 - \sigma_3 ^{-2}$ , thus $ \sigma_2 ,   \sigma_1 
\in F ( \sigma_3)$
hence $ F ( \sigma_3) = K$.

\end{proof}
\begin{theo}
Let $\sC$  be the connected component  of
the moduli space constituted by the primary Burniat surfaces ($K^2_S = 6$).

Then  $\sC$  is a rational 4-fold.
\end{theo}
\begin{proof}
As in the proof of the previous theorem, we have to divide a parameter space
$\cong \CC^6$, parametrizing three pairs of lines of
equations  $ x_{i+2} = a_i x_{i+1} , x_{i+2} = b_i x_{i+1}  $
by the action of $({\CC^*})^2$, of $\mathfrak S_3$, of the $ (\ZZ/2 \ZZ)^3$
  generated by the transformations  $\tau_i$ such that
$\tau_i$ exchanges $a_i$ with $b_i$, and finally of the Cremona
transformation (mapping $a_i$ to $a_i^{-1}$, $b_i$ to $b_i^{-1}$).

As before, we can replace the action of $\mathfrak S_3$ by the direct 
sum of two copies of the
standard permutation representation (of the $a_i$'s and of the $b_i$'s ).

Moreover, we have the action of the subgroup $({\CC^*})^2 \subset \PP 
GL ( 3, \CC)$
of diagonal matrices  $({\CC^*})^2 : = \{ diag (t_1, t_2, t_3 )| t_i 
\in \CC^* \}$ :
$$
a_i \mapsto a_i \frac{t_{i+1}}{t_{i+2}}, \ b_i \mapsto b_i 
\frac{t_{i+1}}{t_{i+2}}, \ i \in \{1,2,3 \}.
$$

We set: $\lambda_i := \frac{t_{i+1}}{t_{i+2}}$, thus $\prod_i 
\lambda_i = 1$ and our
$({\CC^*})^2$ is the  subgroup of $({\CC^*})^3$,
$$({\CC^*})^2 =\{  ( \lambda_1,  \lambda_2, \lambda_3  )
| \prod_i \lambda_i = 1 \}.$$

The invariants for the $ (\ZZ/2 \ZZ)^3$-action are: $u_i:= a_ib_i$, 
$v_i:=a_i + b_i$
and $({\CC^*})^3$ acts by
$$ u_i \mapsto \lambda_i  ^2 u_i , \ v_i \mapsto \lambda_i   v_i  .$$

\begin{claim}
The invariants for the $({\CC^*})^2$-action  are
$$w_i:= \frac{u_i}{v_i^2}, \ i =1,2,3; \ v:= \prod_{i=1}^3 v_i.
$$
\end{claim}

\begin{proof}[Proof of the claim]
Clearly the field of $({\CC^*})^3$ -invariants is generated by the $w_i$'s,
and we can replace the generators $u_i, v_i$ ($i=1,2,3$) by the
generators $w_i, v_i$ ($i=1,2,3$).

Now the exact sequence of algebraic groups
$$1 \ra   ({\CC^*})^2 \ra ({\CC^*})^3 \ra {\CC^*} \ra 1$$
where $( \lambda_1,  \lambda_2, \lambda_3  ) \mapsto \lambda : = 
\prod_i \lambda_i$,
shows that the projection $({\CC^*})^3 \ra {\CC^*}$  is the quotient 
map by the $({\CC^*})^2$ action.

Since $\CC
(v_1, v_2, v_3)$ is the function field of $({\CC^*})^3$, the field of 
invariants for $({\CC^*})^2$ acting
on $\CC (v_1, v_2, v_3)$ is  $\CC (v)$.

\end{proof}
Note that $\mathfrak S_3$ acts on $\{w_1,w_2,w_3 \}$ by the 
permutation representation, whereas the
Cremona transformation acts by $w_i \mapsto w_i$, $v \mapsto 
\frac{\prod v_i}{\prod u_i} =:\frac{v}{u}$. In
fact, the Cremona transformation sends
$u_i$ to
$u_i^{-1}$ and $v_i$ to $\frac{1}{a_i} + \frac{1}{b_i} = 
\frac{a_i+b_i}{a_ib_i} = \frac{v_i}{u_i}$.

Since $\frac{u}{v^2} = \prod w_i$ it follows that $u=\prod w_i v^2$, 
thus $v \mapsto (\prod w_i)^{-1} v^{-1}$.
The invariants for the Cremona transformations are therefore
$$w_1,w_2,w_3, v+ \frac{1}{v \prod w_i}$$
where the last element is $\mathfrak S_3$-invariant.

Finally,   the invariants for the action of $\mathfrak S_3$ are: the 
elementary symmetric
functions $\sigma_1(w_1,w_2,w_3)$, $\sigma_2(w_1,w_2,w_3)$, 
$\sigma_3(w_1,w_2,w_3)$,
and $v+ \frac{1}{v
\prod w_i}$.

Thus the field of invariants is rational.

\end{proof}

\bigskip

We derive now some easy consequences of the main theorems:

\begin{cor}
All surfaces $S$ which are deformations of Burniat surfaces with  $ 
K^2_S \geq 4$ are
again Burniat surfaces, and the bicanonical map of their canonical 
model is a finite
morphism $ \Phi_2 \colon X \ra Y'$, Galois with group $G = (\ZZ/2 \ZZ)^2$,
and with image a Del Pezzo surface $Y'$ of degree $ K^2_S $.
$Y'$  is singular  exactly for the nodal familys
with $ K^2_S = 4$ (it has precisely one $A_1$ singularity then).
In particular, Bloch's conjecture $A_0(S) = \ZZ$ holds for for all the surfaces
in these 4 connected components of the moduli space.
\end{cor}

\begin{proof}
The last statement follows from our main theorems and the work of 
Inose and Mizukami \cite{im}.

In fact Inose and Mizukami show that Bloch's conjecture holds for 
certain classes of
Inoue surfaces, which we have shown in part one (\cite{burniat1}) to 
coincide with
the classes of Burniat surfaces.

\end{proof}


\smallskip
\noindent {\bf Authors' Addresses:}\\
\noindent I.Bauer, F. Catanese, \\ Lehrstuhl Mathematik VIII,\\ 
Mathematisches Institut der Universit\"at
Bayreuth\\ NW II,  Universit\"atsstr. 30\\ 95447 Bayreuth\\
  email: ingrid.bauer@uni-bayreuth.de,
         fabrizio.catanese@uni-bayreuth.de

\end{document}